\documentclass[a4paper,12pt,oneside]{article}
\usepackage[left=1.5cm,right=1.5cm,top=3cm,bottom=3cm]{geometry}

\usepackage[latin1]{inputenc}
\usepackage{amssymb, amsmath, tikz, stmaryrd, amsfonts, latexsym, amscd, amsthm, enumitem, epstopdf, graphicx, caption, float, multirow, url, epstopdf, xcolor, enumerate, fancyhdr, afterpage}
\usepackage[normalem]{ulem}
\usepackage[backref]{hyperref}
\usepackage[alphabetic,backrefs,lite]{amsrefs}
\usepackage[all]{xy}

\DeclareGraphicsExtensions{.pdf,.eps,.png,.jpg,.mps}

\newtheorem*{theorem*}{Theorem}
\newtheorem*{corollary*}{Corollary}

\newtheorem{theorem}{Theorem}[section]
\newtheorem{lemma}[theorem]{Lemma}
\newtheorem{proposition}[theorem]{Proposition}
\newtheorem{corollary}[theorem]{Corollary}

\theoremstyle{definition}

\newtheorem{remark}[theorem]{Remark}

\newtheorem{question}[theorem]{Question}

\makeatletter
\renewcommand\@biblabel[1]{\hfill#1.}
\makeatother

\begin{document}
\title{Images of Galois representations in mod $p$ Hecke algebras}
\author{Laia Amor\'{o}s}

\newcommand{\Addresses}{{
\bigskip
\footnotesize
Laia Amor\'os, \textsc{Department of Mathematics and Systems Analysis, Otakaari 1, Espoo, FI-00076 Finland}\par\nopagebreak
  \textit{E-mail address} \texttt{laia.amoros@aalto.fi}
}}

\date{}

\maketitle

\abstract{
\small
Let $(\mathbb{T}_f,\mathfrak{m}_f)$ denote the mod $p$ local Hecke algebra attached to a normalised Hecke eigenform $f$, which is a commutative algebra over some finite field $\mathbb{F}_q$ of characteristic $p$ and with residue field $\mathbb{F}_q$. By a result of Carayol we know that, if the residual Galois representation $\overline{\rho}_f:G_\mathbb{Q}\rightarrow\mathrm{GL}_2(\mathbb{F}_q)$ is absolutely irreducible, then one can attach to this algebra a Galois representation $\rho_f:G_\mathbb{Q}\rightarrow\mathrm{GL}_2(\mathbb{T}_f)$ that is a lift of $\overline{\rho}_f$. We will show how one can determine the image of $\rho_f$ under the assumptions that $(i)$ the image of the residual representation contains $\mathrm{SL}_2(\mathbb{F}_q)$, $(ii)$ $\mathfrak{m}_f^2=0$ and $(iii)$ the coefficient ring is generated by the traces. As an application we will see that the methods that we use allow to deduce the existence of certain $p$-elementary abelian extensions of big non-solvable number fields.
}

\smallskip


\section{Introduction}

Fix a weight $k\geq2$ and a level $N\geq1$. Denote by $S_k(N;\mathbb{C})$ the complex vector space of cusp forms of weight $k$, level $N$ and trivial character.
The space $S_k(N;\mathbb{C})$ is furnished with Hecke operators $T_p$ for every prime $p$.
These operators commute with one another and they generate a finite-dimensional commutative $\mathbb{Z}$-subalgebra inside $\mathrm{End}_\mathbb{C}(S_k(N;\mathbb{C}))$, called the \textit{Hecke algebra} of $S_k(N;\mathbb{C})$ and denoted by $\mathbb{T}_k(N)$.
A modular form $f=\sum_{n\geq0}a_n(f)q^n\in\mathrm{S}_k(N;\mathbb{C})$ that is a simultaneous eigenvector for all Hecke operators is called an \textit{eigenform} or \textit{Hecke eigenform}. It is said to be \textit{normalised} if $a_1(f)=1$.

Let $S_k(N;\mathbb{Z})$ denote the abelian group of cusp forms with coefficients in $\mathbb{Z}$.
For any $\mathbb{Z}$-algebra $\sigma:\mathbb{Z}\rightarrow R$, we let $S_k(N;R):=S_k(N;\mathbb{Z})\otimes_\mathbb{Z}R$.
There is a natural isomorphism of $\mathbb{Z}$-modules
$$\begin{array}{rccc}
\lambda:& S_k(N;R) & \rightarrow & \mathrm{Hom}_\mathbb{Z}(\mathbb{T}_k(N), R) \\
& f & \mapsto & \lambda_f: T_n\mapsto a_n(f).
\end{array}$$
Let $\mathbb{T}_R:=\mathbb{T}_k(N)\otimes_\mathbb{Z}R$. Then we also have an isomorphism 
$\mathrm{Hom}_\mathbb{Z}(\mathbb{T}_k(N), R)\simeq\mathrm{Hom}_R(\mathbb{T}_R, R)$ (cf. \cite{Wiese06}).
Thus, every element $f\in S_k(N;R)$ corresponds to a linear function $\Phi:\mathbb{T}_R\rightarrow R$ and is uniquely identified by its formal $q$-expansion $f=\sum_n\Phi(T_n)q^n=\sum_na_n(f)q^n$,

Let $R$ be a complete local ring with maximal ideal $\mathfrak{m}$ and residual field $\mathbb{F}$ of characteristic $p$.
Let $f\in\mathrm{S}_k(N;R)$ denote a normalised Hecke eigenform given by a ring homomorphism $\lambda_f:\mathbb{T}_k(N)\rightarrow R$.
Let $\overline{f}\in\mathrm{S}_k(N;\mathbb{F})$ denote the residual form with coefficients in $\mathbb{F}$, which corresponds to a homomorphism
$\lambda_{\overline{f}}=\overline{\lambda}_f:\mathbb{T}_k(N)\rightarrow\mathbb{F}$, $T_n \mapsto  a_n(f)\mbox{ mod }\mathfrak{m}$,
given by the reduction modulo the maximal ideal of $R$. In fact, $f$ has coefficients in some finite field $\mathbb{F}_q\subseteq\mathbb{F}$. Then it is well-known (see for example \cite{Deligne-Serre74}, Theorem 6.7) that we can attach to $\overline{f}$ a residual Galois representation
$\overline{\rho}_f:G_\mathbb{Q}\rightarrow\mathrm{GL}_2(\mathbb{F}_q)$
which is continuous, unramified outside $Np$ and, for every $\ell\nmid Np$, the following relations hold
$$\mathrm{tr}(\overline{\rho}_f(\mathrm{Frob}_\ell))=\overline{\lambda}_f(T_\ell)\quad\mbox{and}\quad\mathrm{det}(\overline{\rho}_f(\mathrm{Frob}_\ell))=\ell^{k-1}.$$

Fix some normalised Hecke eigenform $f(z)=\sum_{n=0}^\infty a_n(f)q^n\in\mathrm{S}_k(N;\mathbb{C})$, whose coefficients lie in some ring of integers of a number field, and let $\overline{\rho}_f:G_\mathbb{Q}\rightarrow\mathrm{GL}_2(\overline{\mathbb{F}}_p)$ denote the attached mod $p$ Galois representation described above.
Let $\mathbb{F}_q$ be the finite field generated by the $a_\ell(f)$ mod $p$. Let $\mathbb{T}_{\mathbb{F}_q}:=\mathbb{T}_\mathbb{Z}\otimes_\mathbb{Z}\mathbb{F}_q$ and denote by $\overline{\mathbb{T}}\subseteq\mathbb{T}_{\mathbb{F}_q}$ the $\mathbb{F}_q$-subalgebra generated by the operators $T_\ell$ with $\ell\nmid Np$. Consider the ring homomorphism $\overline{\lambda}_f: \overline{\mathbb{T}} \rightarrow \mathbb{F}_q$ and let $\mathfrak{m}_f:=\mbox{ker}(\overline{\lambda}_f)$, which is a maximal ideal of $\overline{\mathbb{T}}$.
We denote by $\overline{\mathbb{T}}_{\mathfrak{m}_f}$ the localisation of $\overline{\mathbb{T}}$ at $\mathfrak{m}_f$. Then $\overline{\mathbb{T}}_{\mathfrak{m}_f}$ is a commutative local finite-dimensional $\mathbb{F}_q$-algebra
and, if moreover we assume that the residual Galois representation $\overline{\rho}_f$ is absolutely irreducible, by Theorem 3 in \cite{Carayol94} we obtain a continuous Galois representation
$$\rho_f:G_\mathbb{Q}\rightarrow\mathrm{GL}_2(\overline{\mathbb{T}}_{\mathfrak{m}_f})$$
which is unramified outside $Np$ and, for every $\ell\nmid Np$, one has:
$$\mathrm{tr}(\rho_f(\mathrm{Frob}_\ell))=T_\ell\quad\mbox{and}\quad\mathrm{det}(\rho_f(\mathrm{Frob}_\ell))=\ell^{k-1}.$$
This representation is unique up to conjugation 
and $\overline{\rho}_f=\pi\circ\rho_f$, where $\pi$ extends the natural projection $\pi:\overline{\mathbb{T}}_{\mathfrak{m}_f}\rightarrow\mathbb{F}_q$. In this setting, one is naturally interested in the image of $\rho_f$.

We will be studying the case $\mathrm{Im}(\overline{\rho}_f)=\mathrm{GL}_2^D(\mathbb{F}_q)=\{g\in\mathrm{GL}_2(\mathbb{F}_q)\mid\det(g)\in D\}$, where $D$ is a subgroup of $\mathbb{F}_q^\times$. Let  $\mathrm{M}_2^0(\mathbb{F}_q)$ denote the trace-0 matrices with coefficients in $\mathbb{F}_q$, denote by $\mathbb{S}$ the subspace of scalar matrices in $\mathrm{M}_2^0(\mathbb{F}_q)$ and let $\mathrm{GL}^D_2(\mathbb{T}):=\{g\in\mathrm{GL}_2(\mathbb{T}) \mid \det(g)\in D \}$, where $\mathbb{T}$ denotes a finite dimensional local commutative $\mathbb{F}_q$-algebra.
In the present paper we will prove the following result.

\begin{theorem*}
	Let $\mathbb{F}_q$ denote a finite field of characteristic $p$ and $q=p^d$ elements, and suppose that $q\neq2,3,5$. Let $(\mathbb{T},\mathfrak{m}_{\mathbb{T}})$ be a finite-dimensional local commutative $\mathbb{F}_q$-algebra equipped with the discrete topology, and with residue field $\mathbb{T}/\mathfrak{m}_\mathbb{T}\simeq\mathbb{F}_q$. Suppose that $\mathfrak{m}_\mathbb{T}^2=0$. Let $\Pi$ be a profinite group and let $\rho:\Pi\rightarrow\mathrm{GL}_2(\mathbb{T})$ be a continuous representation such that
	\begin{itemize}
		\item[$(a)$] $\mathrm{Im}(\rho)\subseteq\mathrm{GL}_2^D(\mathbb{T})$, where $D\subseteq\mathbb{F}_q^\times$ is a subgroup.
		\item[$(b)$] $\mathrm{Im}(\overline{\rho})=\mathrm{GL}_2^D(\mathbb{F}_q)$, where $\overline{\rho}$ denotes the reduction $\rho\ \mathrm{mod}\ \mathfrak{m}_\mathbb{T}$.
		\item[$(c)$] $\mathbb{T}$ is generated as $\mathbb{F}_q$-algebra by the set of traces of $\rho$.
	\end{itemize}
	Let $m:=\mathrm{dim}_{\mathbb{F}_q}\mathfrak{m}_\mathbb{T}$ and let $t$ be the number of different traces in $\mathrm{Im}(\rho)$.
	\begin{itemize}
		\item[$(i)$] If $p\neq2$, then $t=q^{m+1}$ and
	$$\mathrm{Im}(\rho)\simeq(\underbrace{\mathrm{M}_2^0(\mathbb{F}_q)\oplus\ldots\oplus\mathrm{M}_2^0(\mathbb{F}_q)}_m)\rtimes\mathrm{GL}_2^D(\mathbb{F}_q)\simeq\mathrm{GL}_2^D(\mathbb{T}).$$
		\item[$(ii)$] If $p=2$, then
		$t=q^\alpha\cdot((q-1)2^\beta+1),\mbox{ for some unique }0\leq \alpha\leq m\mbox{ and }0\leq\beta\leq d(m-\alpha)$,
	and in this case $\mathrm{Im}(\rho)\simeq M\rtimes\mathrm{GL}^D_2(\mathbb{F}_q)$, where $M$ is an $\mathbb{F}_2[\mathrm{GL}^D_2(\mathbb{F}_q)]$-submodule of $\mathrm{M}_2^0(\mathfrak{m}_\mathbb{T})$ of the form
	$$M\simeq\underbrace{\mathrm{M}_2^0(\mathbb{F}_q)\oplus{\ldots}\oplus\mathrm{M}_2^0(\mathbb{F}_q)}_{\alpha}\oplus\underbrace{C_2\oplus\cdots\oplus C_2}_{\beta},$$
	where $C_2\subseteq\mathbb{S}$ is a subgroup with $2$ elements of the scalar matrices. Moreover, $M$ is determined uniquely by $t$ up to isomorphism.
	\end{itemize}
\end{theorem*}

An interesting consequence of this result is the following.

\begin{corollary*}
	Fix an odd prime $p$ and let $k$ denote a finite field of characteristic $p$ with $\#k>5$. Let $R$ be a complete noetherian local $k$-algebra with residue field $k$.
	Let $\overline{\rho}:G_{\mathbb{Q},Np}\rightarrow\mathrm{GL}_2(k)$ be a continuous representation with $\mathrm{Im}(\overline{\rho})=\mathrm{GL}^D_2(k)$, for some subgroup $D\subseteq k^\times$.
	Consider a deformation
	$\rho:G_{\mathbb{Q},Np}\rightarrow\mathrm{GL}_2(R)$
	of $\overline{\rho}$ with $\det(\rho)=\det(\overline{\rho})$, and assume that $R$ is topologically generated over $k$ by the traces of $\rho$. Then $\mathrm{Im}(\rho)\simeq \mathrm{GL}_2^D(R)$.
\end{corollary*}

This corollary also tells us about the images of Galois representations taking values in the mod $p$ Hecke algebras defined in \cite{BeKh15} and \cite{Deo17}

Another important application of the these representations is that they can give us information about some abelian extensions of number fields. The Galois representations that we investigate correspond to abelian extensions of very big non-solvable number fields that standard methods do not allow to treat computationally.

The structure of this paper is as follows. In Section \ref{sec_semidirect_product} we prove the image-splitting theorem (Theorem \ref{image}), a generalisation of Manoharmayum's Main Theorem in \cite{Ma15} that will allow to express $\mathrm{Im}(\rho_f)$ as a semidirect product of a certain submodule $M\subseteq\mathrm{M}_2^0(\mathfrak{m}_f)$ and $\mathrm{GL}_2^D(\mathbb{F}_q)$. 
In Section \ref{sec_theorem} we give a complete classification of the possible images of $2$-dimensional Galois representations with coefficients in local algebras over finite fields under the hypothesis previously mentioned (Theorem \ref{thm2}). 
These algebras are studied in detail and more generality in \cite{BeKh15} and \cite{Deo17}.
In Section \ref{sec_algorithm} we use this result in the situation where the Galois representation takes values in mod $p$ Hecke algebras coming from modular forms. The classification of the images appearing in this setting is completely solved in odd characteristic. In even characteristic we can only reduce the problem to a finite number of possibilities.
In Section \ref{sec_examples} we see how one can compute explicit examples in characteristic 2 and show how in each case one can have a conjectural image. We then state some questions that arise after contrasting many examples.
Finally in Section \ref{sec_application} we will see that the methods on group extensions that we develop allow  to deduce the existence of $p$-elementary abelian extensions of big number fields, which are not computationally approachable so far.

In \cite{Bel16} and \cite{CLM19} the authors prove a result along the lines of our image-splitting theorem \ref{image}, but considering all possible images for the residual representation $\overline{\rho}$ and in the more general setting of pseudo-representations (but excluding the case $p=2$).
Let $\Pi$ be a profinite group and $A$ a compact local ring with maximal ideal $\mathfrak{m}$ and residue field $\mathbb{F}=A/\mathfrak{m}$ a finite field of characteristic $p\neq2$.
In both papers \cite{Bel16} and \cite{CLM19} the key object of study is a certain \textit{Pink-Lie algebra} $L$ attached to $\mathrm{Im}(\rho)$, which the authors then describe depending on the different possible images of $\overline{\rho}$.
In \cite{Bel16}, the author considers two-dimensional pseudo-representations $(t,d)$ of $\Pi$ over $A$ and attaches to it a \textit{generalized matrix algebra} $R$ over $A$ and a representation $\rho:\Pi\rightarrow R^*$ with trace $t$ and determinant $d$. He proceeds to describe the image $\rho(\Pi)$ of the family $(t,d)$, obtaining a complete description depending on the projective image of the representation $\overline{\rho}$.
In particular, the author proves that the image of $\rho$ contains a nontrivial congruence subgroup of $\mathrm{SL}_2(B)$ for a certain subring $B$ of $A$. 

In \cite{CLM19}, the authors show that the ring $B$ can be slightly enlarged into an optimal ring $A_0$, and they describe it in terms of the \textit{conjugate self-twists} of $\rho$.
They prove, under mild conditions, that if $\overline{\rho}$ is either absolutely irreducible, or a sum of two distinct characters to $\mathbb{F}^\times$, then $\rho$ is $A_0$-full, i.e. that $\mathrm{Im}(\rho)$ contains a congruence subgroup $\Gamma_{A_0}(\mathfrak{a}):=\ker\left(\mathrm{SL}_2(A_0)\rightarrow\mathrm{SL}_2(A_0/\mathfrak{a})\right)$ for some nonzero ideal $\mathfrak{a}$ of $A_0$.
When the image of $\overline{\rho}$ is big, they are able to refine their main result using  Manoharmayum's Main Theorem in \cite{Ma15}.
In \cite{CLM19} Prop 6.7, the authors prove that, if $\mathrm{Im}(\overline{\rho})\supseteq\mathrm{SL}_2(\mathbb{E})$, where $\mathbb{E}$ denotes the residue field of $A_0$, then
$\mathrm{Im}(\rho)$ contains $\mathrm{SL}_2(A_1)$, where $A_1$ is an extension of certain trace-0 diagonal group $I$ by the Witt vectors $W(\mathbb{E})$ of $\mathbb{E}$.

\subsection*{Acknowledgements}
The author would like to thank Gabor Wiese for the guidance and support that led to the realisation of this work. A special thanks to Guillermo Mantilla for the helpful modules discussions. The author would also like to thank the referee for the careful reading and valuable comments, which helped to improve the manuscript.

\section{$\mathrm{Im}(\rho)$ as a semidirect product}\label{sec_semidirect_product}

We will start by proving the image-splitting theorem, a result that generalises the Main Theorem in \cite{Ma15}. As a consequence of this result, we will be able to express, under certain conditions, the image of $\rho_f$ as a semidirect product, a first step to completely determine $\mathrm{Im}(\rho_f)$.

In order to do this, we need to work in a more general situation. Let $(\mathbb{T},\mathfrak{m}_\mathbb{T})$ denote a finite-dimensional commutative local $\mathbb{F}_q$-algebra with residue field $\mathbb{T}/\mathfrak{m}_\mathbb{T}\simeq\mathbb{F}_q$ of characteristic $p$, and denote by $\pi:\mathbb{T}\rightarrow \mathbb{T}/\mathfrak{m}_\mathbb{T}$ the natural projection. Suppose that $\mathfrak{m}_\mathbb{T}^2=0$. Consider a closed subgroup $G\subseteq\mathrm{GL}_2^D(\mathbb{T})$, where $D\subseteq\mathbb{F}_q^\times$ denotes the subgroup where the determinants in $G$ lie. Suppose that $\pi(G)=\mathrm{GL}_2^D(\mathbb{F}_q)$. This gives a short exact sequence
$$1\rightarrow H\rightarrow G\rightarrow\mathrm{GL}_2^D(\mathbb{F}_q)\rightarrow1,$$
where $H$ is an $\mathbb{F}_p[\mathrm{GL}_2^D(\mathbb{F}_q)]$-module inside the abelian group of trace-0 matrices $\mathrm{M}_2^0(\mathfrak{m}_\mathbb{T})$ with coefficients in $\mathfrak{m}_\mathbb{T}\simeq\mathbb{F}_q^r$, for some $r\geq1$. The image-splitting theorem will allow us to describe the group $G$ as a semidirect product $H\rtimes\mathrm{GL}_2^D(\mathbb{F}_q)$. Equivalently, it tells us that the previous short exact sequence admits a split.

Let us introduce some notation necessary to state the image-splitting theorem.
Let $k$ be a finite field of characteristic $p$. Let $W(k)$ denote its ring of Witt vectors and denote by $T:k\rightarrow W(k)\subset\overline{\mathbb{Q}}_p$ the Teichm\"uller lift. Consider a complete local ring $(A,\mathfrak{m}_A)$ with residue field containing $k$ and consider the inclusion $\iota_0:k\rightarrow A/\mathfrak{m}_A$. Since $W(k)$ is a $p$-ring with residue field $k$, by the structure theorem for complete local rings (\cite{Ma86}, Theorem 29.2) we have that there exists a local homomorphism $\iota:W(k)\rightarrow A$ which induces $\iota_0$ on the residue fields. Thus we have a commutative diagram
$$\xymatrix{
	W(k) \ar@{-->}[r]^{\iota} & A  \\
	k\ar[r]_{\iota_0} \ar[u]_T & A/\mathfrak{m}_A  \ar[u]
	}$$
Denote by $W(k)_A$ the image of $\iota:W(k)\rightarrow A$. 
Consider a subgroup $D\subseteq k^\times$ and define the following group
$$\mathrm{GL}_n^D(W(k)):=\{g\in\mathrm{GL}_n(W(k))\mid\det(g)\in T(D)\}.$$
For any subring $X\subseteq A$, we define the group
$$\mathrm{GL}_n^D(X):=\{g\in\mathrm{GL}_n(X)\mid\det(g)\in \iota(T(D))\}.$$
We will prove the following result.

\begin{theorem}[Image-splitting theorem]\label{image}\label{main thm}
	Let $(A,\mathfrak{m}_A)$ be a complete local noetherian ring with maximal ideal $\mathfrak{m}_A$ and finite residue field $A/\mathfrak{m}_A$ of characteristic $p$. Let $\pi:A\rightarrow A/\mathfrak{m}_A$ denote the natural projection. Suppose we are given a subfield $k$ of $A/\mathfrak{m}_A$ and a closed subgroup $G$ of $\mathrm{GL}_n(A)$. Assume that the cardinality of $k$ is at least $4$ and that $k\neq\mathbb{F}_5$ if $n=2$ and $k\neq\mathbb{F}_4$ if $n=3$. Suppose that $\pi(G)\supseteq\mathrm{GL}_n^D(k)$. Then $G$ contains a conjugate of $\mathrm{GL}_n^D(W(k)_A)$.
\end{theorem}

This result reduces to Manoharmayum's Main Theorem in (\cite{Ma15}) when restricted to the case $D=1$. More precisely, in \cite{Ma15} the author proves under the same hypothesis that, if $\pi(G)$ contains $\mathrm{SL}_n(k)$, then $G$ contains a conjugate of $\mathrm{SL}_n(W(k)_A)$.

\begin{proof}
	To begin with, we may assume without loss of generality that $\pi(G)=\mathrm{GL}_n^D(k)$. Indeed, consider $\overline{G}:=G\cap\pi^{-1}(\mathrm{GL}_n^D(k))$. Then $\pi(\overline{G})=\mathrm{GL}_n^D(k)$. If we see that $\overline{G}$ contains a conjugate of $\mathrm{GL}_n^D(W_A)$, then this conjugate also lies in $G$.
	Using the main result of Manoharmayum \cite{Ma15}, by conjugating $G$, we can from now on assume that $G$ itself contains $\mathrm{SL}_n(W(k)_A)$.
	
	Observe that the kernel of $\pi:\mathrm{GL}_n^D(A)\rightarrow\mathrm{GL}_n^D(k)$ is a closed pro-$p$ subgroup of $\mathrm{GL}_n^D(A)$. The following result is proved in \cite{wilson}, Proposition 2.3.3.
	
	\begin{lemma}[Profinite Schur-Zassenhaus theorem]\label{pro-finite}
		Let $H$ be a profinite group with  a normal closed pro-$p$ subgroup $P$ such that $H/P$ is finite of order prime to $p$. Then there exists a finite subgroup $A$ of $H$ that maps isomorphically to $H/P$, and any two such subgroups $A, A'$ of $H$ are conjugate by an element of $H$.
	\end{lemma}
	
	In what follows, we will call a group $A$ in $H$ as in the previous lemma a \emph{complement} of $P$.
	Let $T^D_n(k)$ be the set of diagonal matrices in $\mathrm{GL}_n(k)$ whose determinant lies in $D$. 	Clearly $T^D_n(k)$ lies in $\mathrm{GL}^D_n(k)$. Let $H^D=\pi^{-1}(T^D_n(k))$, which is a closed subgroup of $G$. Let $A^D$ be the set of diagonal matrices in $\mathrm{GL}_n^D(W(k)_A)$ whose entries are Teichm\"uller lifts of the elements in $T^D_n(k)$. When $D$ is the trivial group we just use the superscript $1$ instead of $\{1\}$.

	Let $P$ be the inverse image in $G$ under $\pi$ of the trivial group.
	Observe that $A^1$ lies in $G$ by our hypothesis that $G$ contains $\mathrm{SL}_n(W(k)_A)$. In particular, $A^1$ inside $H^1$ is a complement of $P$. If we can show that $A^D$ is contained in $G$, then the proof of Theorem \ref{main thm} is complete, since $\mathrm{GL}_n^D(W(k)_A)$ is generated by $A^D$ and $\mathrm{SL}_n(W(k)_A)$.
	This is proved in Lemma \ref{complement}.
\end{proof}

\begin{lemma}\label{complement}
	Under the previous notation, $A^D$ is contained in $H^D$.
\end{lemma}

\begin{proof}
 	Let $A_D$ in $H^D$ be a complement of $P$ according to Lemma \ref{pro-finite}. Let $A_1$ be the inverse image of $T^1_n(k)$ inside $A_D$, so that $A_1$ is the complement of $P$ in $H^1$. By Lemma \ref{pro-finite}, there exists $g\in H^1$ such that $A_1$ and $A^1$ are conjugate. Hence, after conjugating this by $g$, we can assume that $A_1=A^1$ and that $A_D$ contains $A^1$.
	Now $A_D$ is commutative, and hence it lies in the centralizer of $A^1$.
	
	Under the hypothesis on $k$ and $n$ in Theorem \ref{main thm}, one can see that the centralizer subgroup of $A^1$ is the set of diagonal matrices in $\mathrm{GL}_n^D(A)$, as the entries of the elements in $A^1$ are Teichm\"uller lifts. In particular, $A_D$ lies in this subgroup.
	
	Finally, for a diagonal matrix in $\mathrm{GL}_n^D(A)$, any entry has finite order prime to $p$ if and only if it is a Teichm\"uller lift of its reduction. It follows that $A_D=A^D$, and this completes the proof of the lemma.
\end{proof}

\begin{corollary}\label{cor}
	Let $k$ be a finite field of characteristic $p$ with cardinality at least $4$, $k\neq\mathbb{F}_5$ if $n=2$ and $k\neq\mathbb{F}_4$ if $n=3$.
	Let $(A,\mathfrak{m}_A)$ be a finite-dimensional commutative local $k$-algebra with residue field $k$ and $\mathfrak{m}_A^2=0$.
	Let $G\subseteq\mathrm{GL}^D_n(A)$ be a subgroup. Suppose that $G\mod\mathfrak{m}_A=\mathrm{GL}_n^D(k)$. Then there exists an $\mathbb{F}_p[\mathrm{GL}_n^D(k)]$-submodule $M\subseteq\mathrm{M}_n^0(\mathfrak{m}_A)$ such that $G$ is, up to conjugation by an element $u\in\mathrm{GL}_n(A)$ with $\pi(u)=1$, a $\mathrm{(}$non-twisted$\mathrm{)}$ semidirect product of the form
	$$G\simeq M\rtimes \mathrm{GL}_n^D(k).$$
\end{corollary}

\begin{proof}
	We are in the hypothesis of the image-splitting theorem, and since moreover we are assuming that $A$ is a $k$-algebra, we have that $W(k)_A=k$. So in this situation there exists $u\in\mathrm{GL}_n(A)$ such that $\pi(u)=1$ and $G\supseteq u^{-1}\mathrm{GL}_n^D(k)u$. Let $G':=uGu^{-1}\subseteq\mathrm{GL}_n^D(A)$. Denote by $\pi:A\rightarrow A/\mathfrak{m}_A\simeq k$ the natural projection and consider the following split short exact sequence:
	$$\begin{array}{ccccccccc}
	0 & \rightarrow & \mathrm{M}_n(\mathfrak{m}_A) & \rightarrow & \mathrm{GL}_n(A) & \stackrel{\pi}{\rightarrow} & \mathrm{GL}_n(k) & \rightarrow & 1 \\
	& & m & \mapsto & 1+m. \\
	\end{array}$$
	Take $m=(m_{ij})_{1\leq i,j\leq n}\in\mathrm{M}_n(\mathfrak{m}_A)$. An easy computation shows that $1+m\in\mathrm{GL}_n^D(A)$ if and only if $\mathrm{tr}(m)=0$:
	$$\det(1+m)=\det\left(\begin{smallmatrix}1+m_{11}&\ldots&m_{1n} \\ \vdots&&\vdots \\ m_{n1}&\ldots&1+m_{nn}\end{smallmatrix}\right)=1+\mathrm{tr}(m),$$
	as $\mathfrak{m}_A^2=0$, and $1+\mathrm{tr}(m)\in k$ if and only if $\mathrm{tr}(m)=0$. This gives a split short exact sequence
	$$\begin{array}{ccccccccc}
	0 & \rightarrow & \mathrm{M}^0_n(\mathfrak{m}_A) & \rightarrow & \mathrm{GL}^D_n(A) & \stackrel{\pi}{\rightarrow} & \mathrm{GL}^D_n(k) & \rightarrow & 1
	\end{array}$$
	and we have a commutative diagram
	$$\begin{array}{ccccccccc}
		0 & \rightarrow & \mathrm{M}_n^0(\ker\pi) & \rightarrow & \mathrm{GL}^D_n(A) & \stackrel{\pi}{\rightarrow} & \mathrm{GL}_n^D(k) & \rightarrow & 1 \\
		& & \cup & & \cup & & || \\
		0 & \rightarrow & M & \rightarrow & G' & \rightarrow & \mathrm{GL}_n^D(k) & \rightarrow & 1.
	\end{array}$$
	with $G'\supseteq\mathrm{GL}^D_n(k)$ and $M\subseteq\mathrm{M}_n^0(\ker\pi)$ a $\mathrm{GL}_n^D(k)$-submodule. So in particular, the second short exact sequence splits and $G'$ is of the form $G'=M\rtimes\mathrm{GL}_n^D(k)$.
\end{proof}

\section{Explicit description of $\mathrm{Im}(\rho)$}\label{sec_theorem}
We will study continuous odd Galois representations
$$\rho:G_\mathbb{Q}\rightarrow\mathrm{GL}_2(\mathbb{T})$$
where $(\mathbb{T},\mathfrak{m}_\mathbb{T})$ denotes a finite-dimensional local commutative algebra over a finite field $\mathbb{F}_q$ of characteristic $p$, equipped with the discrete topology and with $\mathbb{F}_q$ as residue field. We assume that $\mathbb{T}$ is generated by the traces of the image of $\rho$ and that $\mathfrak{m}_\mathbb{T}^2=0$.

We will use the image-splitting theorem to determine, under the hypothesis that $\rho$ has big residual image, the image of $\rho$. More concretely, if $\overline{\rho}:G_\mathbb{Q}\rightarrow\mathrm{GL}_2(\overline{\mathbb{F}}_p)$ denotes the corresponding residual Galois representation, we will assume that $\mathrm{Im}(\overline{\rho})=\mathrm{GL}_2^D(\mathbb{F}_q)$, where $D=\mathrm{Im}(\det(\overline{\rho}))\subseteq\mathbb{F}_q^\times$. Then we will show that the number $t$ of different traces in $\mathrm{Im}(\rho)$ and the dimension $m$ of $\mathfrak{m}_\mathbb{T}$ determine uniquely, up to isomorphism, the group $\mathrm{Im}(\rho)$.

\begin{theorem}\label{thm1}\label{thm2}
	Let $\mathbb{F}_q$ denote a finite field of characteristic $p$ and $q=p^d$ elements, and suppose that $q\neq2,3,5$. Let $(\mathbb{T},\mathfrak{m}_{\mathbb{T}})$ be a finite-dimensional local commutative $\mathbb{F}_q$-algebra equipped with the discrete topology, and with residue field $\mathbb{T}/\mathfrak{m}_\mathbb{T}\simeq\mathbb{F}_q$. Suppose that $\mathfrak{m}_\mathbb{T}^2=0$. Let $\Pi$ be a profinite group and let $\rho:\Pi\rightarrow\mathrm{GL}_2(\mathbb{T})$ be a continuous representation such that
	\begin{itemize}
		\item[$(a)$] $\mathrm{Im}(\overline{\rho})=\mathrm{GL}_2^D(\mathbb{F}_q)$, where $\overline{\rho}$ denotes the reduction $\rho\ \mathrm{mod}\ \mathfrak{m}_\mathbb{T}$ and $D:=\mathrm{Im}(\det\circ\overline{\rho})$,
		
		\item[$(b)$] $\mathrm{Im}(\rho)\subseteq\mathrm{GL}_2^D(\mathbb{T})$,
				
		\item[$(c)$] $\mathbb{T}$ is generated as $\mathbb{F}_q$-algebra by the set of traces of $\rho$.
	\end{itemize}
	Let $m:=\mathrm{dim}_{\mathbb{F}_q}\mathfrak{m}_\mathbb{T}$ and let $t$ be the number of different traces in $\mathrm{Im}(\rho)$. 
	
	\noindent If $p\neq2$, then $t=q^{m+1}$ and
	$$\mathrm{Im}(\rho)\simeq(\underbrace{\mathrm{M}_2^0(\mathbb{F}_q)\oplus\ldots\oplus\mathrm{M}_2^0(\mathbb{F}_q)}_m)\rtimes\mathrm{GL}_2^D(\mathbb{F}_q)\simeq\mathrm{GL}_2^D(\mathbb{T}).$$
	If $p=2$, then
	$t=q^\alpha\cdot((q-1)2^\beta+1),\mbox{ for some }0\leq \alpha\leq m\mbox{ and }0\leq\beta\leq d(m-\alpha)$,
	and in this case $\mathrm{Im}(\rho)\simeq M\rtimes\mathrm{GL}^D_2(\mathbb{F}_q)$, where $M$ is an $\mathbb{F}_2[\mathrm{GL}^D_2(\mathbb{F}_q)]$-submodule of $\mathrm{M}_2^0(\mathfrak{m}_\mathbb{T})$ of the form
	$$M\simeq\underbrace{\mathrm{M}_2^0(\mathbb{F}_q)\oplus{\ldots}\oplus\mathrm{M}_2^0(\mathbb{F}_q)}_{\alpha}\oplus\underbrace{C_2\oplus\cdots\oplus C_2}_{\beta},$$
	where $C_2\subseteq\mathbb{S}$ is a subgroup of order $2$ of the scalar matrices. Moreover, $M$ is determined uniquely by $t$ up to isomorphism.
\end{theorem}

\begin{proof}
	By Cohen Structure Theorem (cf. \cite{Ei95} Theorem 7.7) and the assumption that $\mathfrak{m}_\mathbb{T}^2=0$, we have that $\mathbb{T}\simeq\mathbb{F}_q[X_1,\ldots,X_m]/(X_iX_j)_{1\leq i,j\leq m}$. Let $\pi:\mathbb{T}\rightarrow \mathbb{T}/\mathfrak{m}_\mathbb{T}\simeq\mathbb{F}_q$. Then $\ker\pi=\mathfrak{m}_\mathbb{T}$ is a $\mathbb{F}_q$-vector space of dimension $m$. Put $G:=\mathrm{Im}(\rho)\subseteq\mathrm{GL}^D_2(\mathbb{T})$. 
	We will separate the cases $p\neq2$ and $p=2$.
	
	By the assumptions on $\mathbb{F}_q$ we know (after Corollary \ref{cor}) that, if $p\neq2$, then $G\simeq M\rtimes\mathrm{GL}_2^D(\mathbb{F}_q)$ for some $\mathrm{GL}_2^D(\mathbb{F}_q)$-submodule $M\subseteq\mathrm{M}_2^0(\ker\pi)\simeq\underbrace{\mathrm{M}_2^0(\mathbb{F}_q)\oplus\ldots\oplus\mathrm{M}_2^0(\mathbb{F}_q)}_m$.
	Since $p\neq2$, we are in the case where $\mathrm{M}_2^0(\mathbb{F}_q)$ is a simple $\mathbb{F}_p[\mathrm{GL}_2(\mathbb{F}_q)]$-module, so by Lemma \ref{submodules} we know that $M$ is isomorphic to a direct sum of, a priori, $\alpha\leq m$ copies of $\mathrm{M}_2^0(\mathbb{F}_q)$. 

	Note that if we had $t=q^{\alpha+1}$, with $\alpha < m$, then we could assume, without loss of generality, that the set of traces of $G$ is $T:=\{a_0+a_1X_1+\ldots+a_\alpha X_\alpha\mid a_0,\ldots a_\alpha\in\mathbb{F}_q\}$ and by assumption, $\mathbb{T}=\langle T\rangle$. In particular,  $X_{\alpha+1}\in T$, so $X_{\alpha+1}$ would be a linear combination of $X_1,\ldots,X_\alpha$. On the other hand, we have an isomorphism
	$$\mathrm{GL}_2^D(\mathbb{T})\simeq(\underbrace{\mathrm{M}_2^0(\mathbb{F}_q)\oplus\ldots\oplus\mathrm{M}_2^0(\mathbb{F}_q)}_m)\rtimes\mathrm{GL}_2^D(\mathbb{F}_q),$$
	which gives a contradiction.

	Now suppose that $p=2$. By the assumptions on $\mathbb{F}_q$ we can use Corollary \ref{cor} and obtain that $G\simeq M\rtimes\mathrm{GL}^D_2(\mathbb{F}_q)$ for some $\mathbb{F}_2[\mathrm{GL}^D_2(\mathbb{F}_q)]$-submodule $M\subseteq\mathrm{M}_2^0(\ker\pi)\simeq\mathrm{M}_2^0(\mathbb{F}_q)^m$.
We can apply Lemma \ref{p2} and obtain that
	$$M\simeq N_1\oplus\ldots\oplus N_m,\quad\mbox{with }N_i\subseteq\mathrm{M}_2^0(\mathbb{F}_q),$$
	where each $N_i$ is either $\mathrm{M}_2^0(\mathbb{F}_q)$ or an $\mathbb{F}_2$-subspace of $\mathbb{S}$, for $1\leq i\leq m$. Since the indecomposable subspaces of $\mathbb{S}$ are of the form $C_2\simeq\{0, \lambda\}$ with $\lambda\in\mathbb{S}$, we have $\mathbb{S}\simeq C_2^d$, and
	$$M\simeq\underbrace{\mathrm{M}_2^0(\mathbb{F}_q)\oplus{\ldots}\oplus\mathrm{M}_2^0(\mathbb{F}_q)}_{\alpha}\oplus\underbrace{C_2\oplus\cdots\oplus C_2}_{\beta},\quad 0\leq\alpha\leq m,\ 0\leq\beta\leq d(m-\alpha).$$
	Consider the following split short exact sequence:
	\begin{footnotesize}
	$$\begin{array}{ccccccccc}
	0 & \rightarrow & \mathrm{M}_2^0(\ker\pi) & \stackrel{\iota}{\rightarrow} & \mathrm{GL}^D_2(\mathbb{T}) & \stackrel{\pi}{\rightarrow} & \mathrm{GL}^D_2(\mathbb{F}_q) & \rightarrow & 1 \\
	& & || & & || & & || \\
	0 & \rightarrow & \overbrace{\mathrm{M}_2^0(\mathbb{F}_q)\oplus\ldots\oplus\mathrm{M}_2^0(\mathbb{F}_q)}^m & \rightarrow & \mathrm{GL}^D_2(\mathbb{F}_q[X_1,\ldots,X_m]/(X_iX_j)_{1\leq i,j\leq m}) & \rightarrow & \mathrm{GL}^D_2(\mathbb{F}_q) & \rightarrow & 1 \\ \\
	& & \left(\left(\begin{smallmatrix}a_1&b_1\\c_1&a_1\end{smallmatrix}\right),\ldots, \left(\begin{smallmatrix}a_m&b_m\\c_m&a_m\end{smallmatrix}\right)\right) & \mapsto & 	\left(\begin{smallmatrix}1&0\\0&1\end{smallmatrix}\right)+\left(\begin{smallmatrix}a_1&b_1\\c_1&a_1\end{smallmatrix}\right)X_1+ \ldots +\left(\begin{smallmatrix}a_m&b_m\\c_m&a_m\end{smallmatrix}\right)X_m \\ \\
	& & & & \left(\begin{smallmatrix}a_0&b_0\\c_0&d_0\end{smallmatrix}\right) & \mapsfrom & \left(\begin{smallmatrix}a_0&b_0\\c_0&d_0\end{smallmatrix}\right)
	\end{array}$$
	\end{footnotesize}
	
	\noindent In this setting, after a possible reordering of the variables $X_i$, for an element $\mu\in M$ we have
	$$\iota(\mu)=1+\mu= 1+A_1X_1+\ldots+A_\alpha X_\alpha+ B_1 X_{\alpha+1}+\ldots+B_sX_m,$$
	where
	$$A_k\in N_k\simeq\mathrm{M}_2^0(\mathbb{F}_q),\mbox{ for }1\leq k\leq\alpha,$$
	$$B_k=\left(\begin{smallmatrix} b_k&0 \\ 0&b_k\end{smallmatrix}\right)\in N_{\alpha+k}\simeq\underbrace{C_2\oplus\ldots\oplus C_2}_{e_k}\mbox{ for } 1\leq k\leq s=m-\alpha,\mbox{ and }\sum_{k=1}^s e_k=\beta.$$
	We want to compute the number of different traces in $G$ depending on the module $M$. For an  element $g\in G$ we have $g=(1+\mu)h$ with $\mu\in M$ and $h\in\mathrm{GL}^D_2(\mathbb{F}_q)$, and
	$$\mathrm{tr}(g)=\mathrm{tr}((1+\mu)h)$$
	$$=\mathrm{tr}(h)+\mathrm{tr}(A_1h)X_1+\ldots+\mathrm{tr}(A_\alpha h)X_\alpha+\mathrm{tr}(B_1h)X_{\alpha+1}+\ldots+\mathrm{tr}(B_sh)X_m$$
	$$=\mathrm{tr}(h)+\mathrm{tr}(A_1h)X_1+\ldots+\mathrm{tr}(A_\alpha h)X_\alpha+\mathrm{tr}(h)b_1X_{\alpha+1}+\ldots\mathrm{tr}(h) b_s X_m$$
	$$=\mathrm{tr}(A_1h)X_1+\ldots+\mathrm{tr}(A_\alpha h)X_\alpha+\mathrm{tr}(h)(1+b_1 X_{\alpha+1}+\ldots+b_s X_m).$$
	Let $t$ denote the number of different traces in $G$. Then
	$$t=\#\left\{ \sum_{k=1}^\alpha\mathrm{tr}(A_kh)X_k +\mathrm{tr}(h)(1+\sum_{k=1}^sb_kX_{\alpha+k}) : A_k\in\mathrm{M}_2^0(\mathbb{F}_q), b_k\in C_2^{e_k},\ h\in\mathrm{GL}^D_2(\mathbb{F}_q) \right\}$$
	$$=\#\left\{ \sum_{k=1}^\alpha a_kX_k +\mathrm{tr}(h)(1+\sum_{k=1}^sb_kX_{\alpha+k}) : a_k\in\mathbb{F}_q, b_k\in C_2^{e_k},\ h\in\mathrm{GL}^D_2(\mathbb{F}_q) \right\}$$
	$$=\underbrace{\#\left\{ \sum_{k=1}^\alpha a_kX_k : a_k\in\mathbb{F}_q\right\}}_{t_1}\cdot
	\underbrace{\#\left\{\mathrm{tr}(h)(1+\sum_{k=1}^sb_kX_{\alpha+k}): b_k\in C_2^{e_k}, h\in\mathrm{GL}^D_2(\mathbb{F}_q)\right\}}_{t_2},$$
	where $t_1=q^\alpha$ and
	$$t_2= \#\{0\}+\#\{x\cdot(1+b_{1j_1}X_{\alpha+1}+\ldots+b_{sj_s} X_m):b_{kj_k}\in C_2^{e_k}, x\in\mathbb{F}_q^\times\}$$
	$$=1+(q-1)\cdot 2^{e_1}\cdot\ldots\cdot 2^{e_s}=1+(q-1)\cdot 2^{\sum e_k}=1+(q-1)2^\beta.$$
	So we obtain the formula
	$ t=q^\alpha\cdot(1+(q-1)2^\beta)$.
	
	Finally, let us prove that the number of traces determines the module $M$ up to isomorphism. Let $t'=q^{\alpha'}\cdot(1+(q-1)2^{\beta'})$. Let us check that if $t=t'$ then one has $\alpha=\alpha'$ and $\beta=\beta'$. Recall that $q=2^d$. Suppose that $\alpha\geq\alpha'$. Then
	$$q^\alpha(1+(q-1)2^\beta)=q^{\alpha'}(1+(q-1)2^{\beta'})\Leftrightarrow \underbrace{q^{\alpha-\alpha'}(1+(q-1)2^\beta)}_{(L)}=\underbrace{1+(q-1)2^{\beta'}}_{(R)}.$$
	If $\beta'>0$ then $(R)\equiv1$ mod 2, so $\alpha=\alpha'$ and then $\beta=\beta'$.
	If $\beta'=0$ then $(R)=q$. Since $1+(q-1)2^\beta\geq q$ and $q^{\alpha-\alpha'}\geq1$, we necessarily have $\beta=0=\beta'$ and $\alpha=\alpha'$.
\end{proof}

\begin{remark}
	Theorem $\ref{thm1}$ shows that if the field cut out by $\overline{\rho}$ admits some abelian extension $\mathrm{(}$of the type we are considering$\mathrm{)}$, then it is a ``big one". So the cases $m\geq2$ will occur rarely.
\end{remark}

To complete this section we will use the following result to prove an interesting consequence of Theorem \ref{thm1} about the images of Galois representations taking values in mod $p$ Hecke algebras described in \cite{BeKh15} and \cite{Deo17}.

\begin{proposition}[\cite{Bos86}, Prop. 2]\label{boston}
	Let $R$ be a complete noetherian local ring with maximal ideal $\mathfrak{m}$ such that $R/\mathfrak{m}$ is finite and has characteristic $p\neq2$.
	Let $H$ be a closed subgroup of $\mathrm{SL}_n(R)$ projecting onto $\mathrm{SL}_n(R/\mathfrak{m}^2)$. Then $H=\mathrm{SL}_n(R)$.
\end{proposition}

Let $G_{\mathbb{Q},Np}$ denote the Galois group of a maximal algebraic extension of $\mathbb{Q}$ unramified outside $Np$ and $\infty$.
For background on deformation of representations see \cite{mazur1}, \cite{mazur2}.

\begin{corollary}\label{cor3.4}
	Fix an odd prime $p$ and let $k$ denote a finite field of characteristic $p$ with $\#k>5$. Let $R$ be a complete noetherian local $k$-algebra with residue field $k$.
	Let $\overline{\rho}:G_{\mathbb{Q},Np}\rightarrow\mathrm{GL}_2(k)$ be a continuous representation with $\mathrm{Im}(\overline{\rho})=\mathrm{GL}^D_2(k)$, for some subgroup $D\subseteq k^\times$.
	Consider a deformation
	$$\rho:G_{\mathbb{Q},Np}\rightarrow\mathrm{GL}_2(R)$$
	of $\overline{\rho}$ with $\det(\rho)=\det(\overline{\rho})$, and assume that $R$ is topologically generated over $k$ by the traces of $\rho$. Then $\mathrm{Im}(\rho)\simeq \mathrm{GL}_2^D(R)$.
\end{corollary}

\begin{proof}
	Let $\mathfrak{m}$ denote the maximal ideal of $R$ and consider the representation $\rho':G_{\mathbb{Q},Np}\rightarrow\mathrm{GL}_2(R/\mathfrak{m}^2)$, which just consists of composing $\rho$ with $R/\mathfrak{m}^2$.
	We can apply theorem \ref{thm1} to the algebra $\mathbb{T}=R/\mathfrak{m}^2$ and get that
	$\mathrm{Im}(\rho')\supset\mathrm{SL}_2(R/\mathfrak{m}^2)$.
	Then by Proposition \ref{boston}, $\mathrm{Im}(\rho)\supset\mathrm{SL}_2(R)$. 
	Using that $\det(\rho)=\det(\overline{\rho})$ and the fact that $\mathrm{GL}^D_2(R)$ is generated by $\mathrm{SL}_2(R)$ and the set of diagonal matrices whose determinant lies in $D$,
	we obtain that $\mathrm{Im}(\rho)\simeq\mathrm{GL}_2^D(R)$.
	
\end{proof}

Let $A_k$ be the algebra of endomorphisms of $S_{\leq k}(N;\mathbb{F}_q)=\sum_{i=0}^kS_k(N;\mathbb{F}_q)$ generated by the Hecke operators $T_\ell$ and $S_\ell$ for all $\ell\nmid Np$ and
$A:=\varprojlim A_k$. It is well-known that the algebra $A$ is a complete semi-local ring whose local components correspond bijectively to isomorphism classes of modular Galois representations $\overline{\rho}:G_{\mathbb{Q},Np}\rightarrow\mathrm{GL}_2(\mathbb{F}_q)$.
In \cite{BeKh15} (for level $N=1$) and later in \cite{Deo17} (for level $N\geq1$), the authors study the structure of the local components of $A$.
If $\overline\rho$ is a modular representation satisfying the hypothesis of corollary \ref{cor3.4}, then its deformation to the local component $A_{\overline\rho}$ (given in \cite{Deo17} and \cite{BeKh15}) satisfies the hypotheses of the corollary \ref{cor3.4}.

\section{Application: Computation of images of Galois representations with values in mod $p$ Hecke algebras}\label{sec_algorithm}

In order to apply Theorem \ref{thm2} to Galois representations coming from modular forms, let us recall some notation. Fix a level $N\geq1$, a weight $k\geq2$, a prime $p\nmid N$ and a Dirichlet character $\varepsilon:(\mathbb{Z}/N\mathbb{Z})^\times\rightarrow\overline{\mathbb{F}}_p$. Let $\overline{\mathbb{T}}$ denote the Hecke algebra of $S_k(N,\varepsilon;\mathbb{F}_q)$ generated by the \textit{good} operators $T_\ell$, i.e. those with $\ell\nmid Np$. For a normalised Hecke eigenform $f(z)=\sum_{n=0}^\infty a_n(f)q^n\in S_k(N,\varepsilon;\overline{\mathbb{F}}_p)$, whose coefficients generate the field $\mathbb{F}_q$, let $\mathfrak{m}_f:=\ker\lambda_f$ denote the maximal ideal of $\overline{\mathbb{T}}$ given by the ring homomorphism
$\lambda_f:\overline{\mathbb{T}}\rightarrow\mathbb{F}_q,\ T_n\mapsto a_n(f)$.
Assume that $\mathfrak{m}_f^2=0$.
Consider the local algebra $\mathbb{T}_f:=\overline{\mathbb{T}}_{\mathfrak{m}_f}$, which we will refer to as the \textit{local mod $p$ Hecke algebra} associated to the mod $p$ modular form $f$. Let 
$$\rho_f:G_\mathbb{Q}\rightarrow\mathrm{GL}_2(\mathbb{T}_f)$$ 
be the Galois representation attached to $\mathbb{T}_f$ and let $\overline{\rho}_f:G_\mathbb{Q}\rightarrow\mathrm{GL}_2(\overline{\mathbb{F}}_p)$ denote its reduction modulo the maximal ideal. Assume that $\mathrm{Im}(\overline{\rho})=\mathrm{GL}_2^D(\mathbb{F}_q)$, where $D=\mathrm{Im}(\det\circ\overline{\rho}_f)$. Let $m:=\mathrm{dim}_{\mathbb{F}_q}\mathfrak{m}_f$ and let $t$ denote the number of different traces in $\mathrm{Im}(\rho_f)$. 
Then, using the previous theorem, one can easily deduce that, if $p\neq2$, then $t=q^{m+1}$ and
$$\mathrm{Im}(\rho_f)\simeq(\mathrm{M}_2^0(\mathbb{F}_q)\oplus\ldots\oplus\mathrm{M}_2^0(\mathbb{F}_q))\rtimes\mathrm{GL}_2^D(\mathbb{F}_q).$$
If $p=2$, then $t=q^\alpha\cdot((q-1)2^\beta+1)$, for some $0\leq\alpha\leq m$ and $0\leq\beta\leq d(m-\alpha)$, and in this case
$$\mathrm{Im}(\rho_f)\simeq(\underbrace{\mathrm{M}_2^0(\mathbb{F}_q)\oplus\ldots\oplus\mathrm{M}_2^0(\mathbb{F}_q)}_\alpha\oplus\underbrace{C_2\oplus\ldots\oplus C_2}_\beta)\rtimes\mathrm{GL}_2^D(\mathbb{F}_q),$$
where $C_2\subseteq\mathbb{S}$ is a subgroup of 2 elements.

As one can see, in characteristic 2 the result does not allow to completely determine the image of $\rho_f$.
Using the packages \verb+HeckeAlgebras+ and \verb+ArtinAlgebras+ (which can be found in \cite{WieseWeb}) implemented in Magma, one can compute these algebras, and try to determine the image in concrete cases.

The idea is to compute, for a fixed level, weight and character, all local Hecke algebras on the corresponding space. For each one of them, after ensuring that the conditions of theorem \ref{thm1} are satisfied, one just has to compute Hecke operators up to certain bound, and count the number of different operators that appear. If one computes far enough, this number can give a candidate for the image of the corresponding Galois representation. We will see explicit examples of this procedure in the following section.

\section{Examples in characteristic $2$}\label{sec_examples}

In \cite{thesisLaia} many examples for $p=2$ are listed for $m=1,2$ and $3$. In this setting, since there are several possibilities for the image of $\rho_f$ (which we can distinguish by the number of different traces $\widetilde{t}$ computed), we can only ``guess" the group that we obtain.
However, we will see that the possible numbers of traces that can be obtained in each case are integers enough separated so that it is very likely that we can guess the right number of traces.
Here we will summarise the results obtained in \cite{thesisLaia}. We take $1\leq N\leq 1500$ and $k=2,3$.

\subsubsection*{Examples with $m=1$}
	Here we have
	$t=q^\alpha\cdot((q-1)2^\beta+1)$,  with $0\leq\alpha\leq1$ and $0\leq\beta\leq d(1-\alpha)$. The possible numbers of traces in this setting are summarised in table \ref{traces1}.
	
		\begin{table}[H]
		\centering
		\captionsetup{font=footnotesize}
		\footnotesize{
		\begin{tabular}{cc|c|c|c|} \cline{3-5}
			\multicolumn{2}{c|}{\multirow{2}{*}{$\mathbb{F}_{2^2}$} } & \multicolumn{3}{|c|}{$\beta$} \\  \cline{3-5} 
			 & & 0 & 1 & 2  \\ \cline{1-5}
			\multicolumn{1}{|c|}{\multirow{2}{*}{$\alpha$}} & \multicolumn{1}{|c||}{0} & \textbf{4} & \textbf{7} & \textbf{13} \\ \cline{2-5} 
			\multicolumn{1}{|c|}{} & \multicolumn{1}{|c||}{1} & \textbf{16} & - & - \\ \hline
		\end{tabular}
		\quad
		\begin{tabular}{cc|c|c|c|c|} \cline{3-6}
			\multicolumn{2}{c|}{\multirow{2}{*}{$\mathbb{F}_{2^3}$} } & \multicolumn{4}{|c|}{$\beta$} \\  \cline{3-6} 
			 & & 0 & 1 & 2 & 3 \\ \cline{1-6}
			\multicolumn{1}{|c|}{\multirow{2}{*}{$\alpha$}} & \multicolumn{1}{|c||}{0} & \textbf{8} & \textbf{15} & \textbf{29} & \textbf{57} \\ \cline{2-6} 
			\multicolumn{1}{|c|}{} & \multicolumn{1}{|c||}{1} & \textbf{64} & - & - & - \\ \hline
		\end{tabular}
		\quad
		\begin{tabular}{cc|c|c|c|c|c|} \cline{3-7}
			\multicolumn{2}{c|}{\multirow{2}{*}{$\mathbb{F}_{2^4}$} } & \multicolumn{5}{|c|}{$\beta$} \\  \cline{3-7} 
	 		& & 0 & 1 & 2 & 3 & 4 \\ \cline{1-7}
			\multicolumn{1}{|c|}{\multirow{2}{*}{$\alpha$}} & \multicolumn{1}{|c||}{0} & \textbf{16} & \textbf{31} & \textbf{61} & \textbf{121} & \textbf{241} \\ \cline{2-7} 
			\multicolumn{1}{|c|}{} & \multicolumn{1}{|c||}{1} & \textbf{256} & - & - & - & - \\ \hline
		\end{tabular}
		}
		\caption{Possible number of traces when $m=1$.}\label{traces1}
	\end{table}
	
	Let us describe a concrete example in full detail.
	The first example that we find with $m=1$ and degree $d=2$ is in level $N=67$. In this case, by computing Hecke operators up to $b=1000$, we find $\widetilde{t}=7$. As shown in Table \ref{traces1}, we have the following possibilities for $t$: 4,7,13 or 16. So we can only exclude $t=4$. Let $T^{(1)},\ldots,T^{(7)}$ denote the seven operators that we find. If we take a closer look to them, we observe the following multiplicities for each one:
		$$21\times T^{(1)},\ 31\times T^{(2)},\ 15\times T^{(3)},\ 14\times T^{(4)},\ 39\times T^{(5)},\ 16\times T^{(6)},\ 30\times T^{(7)}.$$
	So if the theoretical number of traces were $t=13$ or 16, it would mean that there are still at least 6 Hecke operators left to appear, which seems highly unlikely.
	
	Just to be even more confident that $t=7$, if we compute up to $b=5000$, we still find $\widetilde{t}=7$ and the following multiplicities:
	$$58\times T^{(1)},\ 114\times T^{(2)},\ 69\times T^{(3)},\ 67\times T^{(4)},\ 185\times T^{(5)},\ 63\times T^{(6)},\ 111\times T^{(7)}.$$
	This means that the group $G$ in this case would be $G\simeq C_2\times\mathrm{SL}_2(\mathbb{F}_4)$.

\subsubsection*{Examples with $m=2$}
	Here we have
	$t=q^\alpha\cdot((q-1)2^\beta+1)$, with $0\leq\alpha\leq2$ and $0\leq\beta\leq d(2-\alpha)$. The possible numbers of traces in this setting are summarised in Table \ref{traces2}.
	
	\begin{table}[H]
	\centering
	\captionsetup{font=footnotesize}
	\footnotesize{
	\noindent \begin{tabular}{cc|c|c|c|c|c|c|} \cline{3-7}
		\multicolumn{2}{c|}{\multirow{2}{*}{$\mathbb{F}_{2^2}$} } & \multicolumn{5}{|c|}{$\beta$} \\  \cline{3-7} 
		 & & 0 & 1 & 2 & 3 & 4 \\ \cline{1-7} 
		\multicolumn{1}{|c|}{\multirow{3}{*}{$\alpha$}} & \multicolumn{1}{|c||}{0} & \textbf{4} & \textbf{7} & \textbf{13} & \textbf{25} & \textbf{49} \\ \cline{2-7} 
		\multicolumn{1}{|c|}{} & \multicolumn{1}{|c||}{1} & \textbf{16} & \textbf{28} & \textbf{52} & - & -  \\ \cline{2-7}
		\multicolumn{1}{|c|}{} & \multicolumn{1}{|c||}{2} & \textbf{64}& - & - & - & -  \\ \hline
	\end{tabular}
	\quad\quad
	\noindent \begin{tabular}{cc|c|c|c|c|c|c|c|c|} \cline{3-9}
		\multicolumn{2}{c|}{\multirow{2}{*}{$\mathbb{F}_{2^3}$} } & \multicolumn{7}{|c|}{$\beta$} \\  \cline{3-9} 
		 & & 0 & 1 & 2 & 3 & 4 & 5 & 6 \\ \cline{1-9} 
		\multicolumn{1}{|c|}{\multirow{3}{*}{$\alpha$}} & \multicolumn{1}{|c||}{0} & \textbf{8} & \textbf{15} & \textbf{29} & \textbf{57} & \textbf{113} & \textbf{225} & \textbf{449} \\ \cline{2-9} 
		\multicolumn{1}{|c|}{} & \multicolumn{1}{|c||}{1} & \textbf{64} & \textbf{120} & \textbf{232} & \textbf{456} & - & - & - \\ \cline{2-9}
		\multicolumn{1}{|c|}{} & \multicolumn{1}{|c||}{2} & \textbf{512}& - & - & - & - & - & - \\ \hline
	\end{tabular}
	
	\begin{tabular}{cc|c|c|c|c|c|c|c|c|c|} \cline{3-11}
		\multicolumn{2}{c|}{\multirow{2}{*}{$\mathbb{F}_{2^4}$} } & \multicolumn{9}{|c|}{$\beta$} \\  \cline{3-11} 
		 & & 0 & 1 & 2 & 3 & 4 & 5 & 6 & 7 & 8 \\ \cline{1-11}
		\multicolumn{1}{|c|}{\multirow{3}{*}{$\alpha$}} & \multicolumn{1}{|c||}{0} & \textbf{16} & \textbf{31} & \textbf{61} & \textbf{121} & \textbf{241} & \textbf{481} & \textbf{916} & \textbf{1921} & \textbf{3841} \\ \cline{2-11} 
		\multicolumn{1}{|c|}{} & \multicolumn{1}{|c||}{1} & \textbf{256} & \textbf{496} & \textbf{976} & \textbf{1936} & \textbf{3856} & - & - & - & - \\ \cline{2-11} 
		\multicolumn{1}{|c|}{} & \multicolumn{1}{|c||}{2} & \textbf{4096} & - & - & - & - & - & - & - & -   \\ \hline
	\end{tabular}
	}
	\caption{Possible number of traces when $m=2$.}\label{traces2}
	\end{table}

\subsubsection*{Examples with $m=3$}
	Here we have $t=q^\alpha\cdot((q-1)2^\beta+1)$, with $0\leq\alpha\leq3$ and $0\leq\beta\leq d(3-\alpha)$.  The possible numbers of traces in this setting are summarised in Table \ref{traces3}.
	
\begin{table}[htp]
\centering
\captionsetup{font=footnotesize}
\footnotesize{
\begin{tabular}{cc|c|c|c|c|c|c|c|c|} \cline{3-9}
	\multicolumn{2}{c|}{\multirow{2}{*}{$\mathbb{F}_{2^2}$} } & \multicolumn{7}{|c|}{$\beta$} \\  \cline{3-9} 
	 & & 0 & 1 & 2 & 3 & 4 & 5 & 6 \\ \cline{1-9} 
	\multicolumn{1}{|c|}{\multirow{4}{*}{$\alpha$}} & \multicolumn{1}{|c||}{0} & \textbf{4} & \textbf{7} & \textbf{13} & \textbf{25} & \textbf{49} & \textbf{97} & \textbf{193} \\ \cline{2-9} 
	\multicolumn{1}{|c|}{} & \multicolumn{1}{|c||}{1} & \textbf{16} &\textbf{28} & \textbf{52} & \textbf{100} & \textbf{196} & - & - \\ \cline{2-9}
	\multicolumn{1}{|c|}{} & \multicolumn{1}{|c||}{2} & \textbf{64} & \textbf{112} & \textbf{208} & - & - & - & - \\ \cline{2-9}
	\multicolumn{1}{|c|}{} & \multicolumn{1}{|c||}{3} & \textbf{256} & - & - & - & - & - & - \\ \hline
\end{tabular}

\begin{tabular}{cc|c|c|c|c|c|c|c|c|c|c|} \cline{3-12}
	\multicolumn{2}{c|}{\multirow{2}{*}{$\mathbb{F}_{2^3}$} } & \multicolumn{10}{|c|}{$\beta$} \\  \cline{3-12}
	& & 0 & 1 & 2 & 3 & 4 & 5 & 6 & 7 & 8 & 9 \\ \cline{1-12}
	\multicolumn{1}{|c|}{\multirow{4}{*}{$\alpha$}} & 0 & \textbf{8} & \textbf{15} & \textbf{29} & \textbf{57} & \textbf{113} & \textbf{225} & \textbf{449} & \textbf{897} & \textbf{1793} & \textbf{3585} \\  \cline{2-12} 
	\multicolumn{1}{|c|}{} & 1 & \textbf{64} & \textbf{120} & \textbf{232} & \textbf{456} & \textbf{904} & \textbf{1800} & \textbf{3592} & - & - & - \\  \cline{2-12}
	\multicolumn{1}{|c|}{} & 2 & \textbf{512} & \textbf{960} & \textbf{1856} & \textbf{3648} & - & - & - & - & - & - \\ \cline{2-12}
	\multicolumn{1}{|c|}{} & 3 & \textbf{4096} & - & - & - & - & - & - & - & - & - \\ \hline
\end{tabular}	

\begin{tabular}{cc|c|c|c|c|c|c|c|c|c|} \cline{3-11}
	\multicolumn{2}{c|}{\multirow{2}{*}{$\mathbb{F}_{2^4}$} } & \multicolumn{9}{|c|}{$\beta$} \\  \cline{3-11}
	& & 0 & 1 & 2 & 3 & 4 & 5 & 6 & 7 & 8 \\ \cline{1-11}
	\multicolumn{1}{|c|}{\multirow{4}{*}{$\alpha$}}  & 0 & \textbf{16} & \textbf{31} & \textbf{61} & \textbf{121} & \textbf{241} & \textbf{481} & \textbf{961} & \textbf{1921} & \textbf{3841} \\ \cline{2-11}
	\multicolumn{1}{|c|}{} & 1 & \textbf{256} & \textbf{496} & \textbf{976} & \textbf{1936} & \textbf{3856} & \textbf{7969 }& \textbf{15376} & \textbf{30736} & \textbf{61456} \\ \cline{2-11}
	\multicolumn{1}{|c|}{} & 2 & \textbf{4096} & \textbf{7936} & \textbf{15616} & \textbf{30976} & \textbf{61696} & - & - & - & - \\ \cline{2-11}
	\multicolumn{1}{|c|}{} & 3 & \textbf{65536} & - & - & - & - & - & - & - & -   \\ \hline
	\multicolumn{2}{c|}{} & \multicolumn{4}{|c|}{$\beta$} & \multicolumn{5}{|c}{} \\  \cline{3-6}
	& & 9 & 10 & 11 & 12 & \multicolumn{5}{|c}{} \\	\cline{1-6}
	\multicolumn{1}{|c|}{$\alpha$} & 0 & \textbf{7681} & \textbf{15361} & \textbf{30721} & \textbf{61441} & \multicolumn{5}{|c}{} \\	\cline{1-6}
\end{tabular}
}
\caption{Possible number of traces when $m=3$.}\label{traces3}
\end{table}

In Table \ref{summary}, we summarise all the examples that we have computed. We separate the cases by the dimension $m$ of the maximal ideal $\mathfrak{m}_f/\mathfrak{m}_f^2$, the degree $d$ of the finite field $\mathbb{F}_q$, and the weight $k$. We run through levels $1\leq N\leq 1500$ and count the number of different examples that we obtain depending on the number of different traces that we find. When we find 2 examples that come form the same modular form (i.e. when we find two examples whose operators are exactly the same) then we denote this in the table by $2\times$.

\begin{table}[H]
\centering
\captionsetup{font=small}
\begin{tabular}{|c|c|c|c|c|c|c|c|} \hline
		$m$ & $d$ & $k$ & $t$ & $M$ & $N=p_1$ & $N=p_1p_2$ & $N=p_1p_2p_3$ \\ \hline\hline
		1 & 2 & 2 & $7$ & $C_2$ & 58 & $2\times15$ & \\ \cline{2-8}
		 & 3 & 2 & 15 & $C_2$ & 32 & $2\times9$ & \\ \cline{2-8}
		 & 4 & 2 & 31 & $C_2$ & 42 & $2\times19$ & \\ \hline\hline
		 
		2 & 2 & 2 & $13$ & $C_2\oplus C_2$ & 1 & 48 & $2\times7$ \\ \cline{2-8}
		& 2 & 3 & $13$ & $C_2\oplus C_2$ & 18 & $2\times3$ & \\ \cline{2-8}
		& 3 & 2 & 29 & $C_2\oplus C_2$ &  & 68 & $2\times8$ \\ \cline{2-8}
		& 3 & 3 & 29 & $C_2\oplus C_2$ & 7 & $2\times11$ &  \\ \cline{2-8}
		& 4 & 2 & 61 & $C_2\oplus C_2$ & & 41 & $2\times2$ \\ \cline{2-8}
		& 4 & 3 & 61 & $C_2\oplus C_2$ & 1 & & \\ \hline\hline
		2 & 2 & 3 & $25$ & $C_2\oplus C_2\oplus C_2$ & 20 & $2\times8$ & \\ \cline{2-8}
		& 3 & 3 & 57 & $C_2\oplus C_2\oplus C_2$ & 29 & $2\times11$ & \\ \cline{2-8}
		& 4 & 3 & 121 & $C_2\oplus C_2\oplus C_2$ & 34 & $2\times7$ & \\ \hline\hline
		2 & 2 & 2 & $28$ & $\mathrm{M}_2^0(\mathbb{F}_q)\oplus C_2$ & 2 &  & \\ \cline{2-8}
		& 3 & 2 & 120 & $\mathrm{M}_2^0(\mathbb{F}_q)\oplus C_2$ & 6 & & \\ \hline\hline

		3 & 2 & 2 & 25 &$C_2\oplus C_2\oplus C_2$ & & & 6 \\ \cline{2-8}
		& 2 & 3 & 25 &$C_2\oplus C_2\oplus C_2$ & & 6 & \\ \cline{2-8}
		& 3 & 2 & 57 & $C_2\oplus C_2\oplus C_2$ & & & 15 \\ \cline{2-8}
		& 4 & 2 &121 & $C_2\oplus C_2\oplus C_2$ & & & 8 \\ \hline\hline
		3 & 2 & 3 & 49 & $C_2\oplus C_2\oplus C_2\oplus C_2$ & & 22 & $2\times1$ \\ \cline{2-8}
		& 3 & 3 & 113 & $C_2\oplus C_2\oplus C_2\oplus C_2$ & & 24 & \\ \cline{2-8}
		& 4 & 3 & 241 & $C_2\oplus C_2\oplus C_2\oplus C_2$ & & 26 & \\ \hline\hline
		3 & 2 & 2 & 52 & $\mathrm{M}_2^0(\mathbb{F}_q)\oplus C_2\oplus C_2$ & & 5 & \\ \hline
\end{tabular}
\caption{Number of examples found}\label{examples1}\label{summary}
\end{table}

\begin{remark}
	Note that computing Hecke operators only up to the Sturm bound does not necessarily give us all the different traces, it only ensures us that we have enough operators to generate the whole algebra. Hence it is not a bound high enough to know how many different traces there are. We can see this in a concrete example. 
	
	For level $N=911$, weight $k=2$ and trivial character, we find an example of Hecke algebra $\mathbb{T}_f$ such that the corresponding maximal ideal $\mathfrak{m}_f$ has dimension $m=1$ over $\mathbb{F}_{16}$ (and $\mathfrak{m}_f^2=0$). Here the Sturm bound $B$ is $B=\frac{kN}{12}\prod_{\ell\mid N, \ell\,\mathrm{prime}}(1+\frac{1}{\ell})=152$.
	If we compute operators up to $b=3000$, we find 116 different operators, and if we compute up to $b=5000$ we find 120 different operators. Hence we see that, in this situation, we need to go much further than the Sturm bound.
	
\end{remark}

After the previous tables, one is led to the following questions.

\begin{question}\label{quest1}
	If $\mathrm{dim}_{\mathbb{F}_q}\mathfrak{m}_f/\mathfrak{m}_f^2=1$, then $G\simeq C_2\times\mathrm{SL}_2(\mathbb{F}_q)$.
\end{question}

\begin{question}\label{quest2} 
	If $\mathrm{dim}_{\mathbb{F}_q}\mathfrak{m}_f/\mathfrak{m}_f^2=2$, then
	$$\mathrm{Im}(\rho_f)\simeq\left\{\begin{array}{l}
		(C_2\oplus C_2)\times\mathrm{SL}_2(\mathbb{F}_q), \mbox{ or } \\
		(C_2\oplus C_2\oplus C_2)\times\mathrm{SL}_2(\mathbb{F}_q), \mbox{ or } \\
		(\mathrm{M}_2^0(\mathbb{F}_q)\oplus C_2)\rtimes\mathrm{SL}_2(\mathbb{F}_q).
	\end{array}\right.$$
\end{question}

\begin{question}
	If $\mathrm{dim}_{\mathbb{F}_q}\mathfrak{m}_f/\mathfrak{m}_f^2=3$, then
	$$\mathrm{Im}(\rho_f)\simeq\left\{\begin{array}{l}
		(C_2\oplus C_2\oplus C_2)\times\mathrm{SL}_2(\mathbb{F}_q),	\mbox{ or }\\
		(C_2\oplus C_2\oplus C_2\oplus C_2)\times\mathrm{SL}_2(\mathbb{F}_q), \mbox{ or } \\
		(\mathrm{M}_2^0(\mathbb{F}_q)\oplus C_2\oplus C_2)\rtimes\mathrm{SL}_2(\mathbb{F}_q).
	\end{array}\right.$$
\end{question}

\section{Existence of $p$-elementary abelian extensions of non-solvable number fields}\label{sec_application}

In this section we use the previous results to predict the existence of $p$-elementary abelian extensions of \textit{very big} non-solvable extensions of $\mathbb{Q}$. Let $\mathbb{F}_q$ denote a finite field of characteristic $p$ and let $\rho:G_\mathbb{Q}\rightarrow\mathrm{GL}_2(\mathbb{T})$ be a Galois representation that takes values in some finite-dimensional local commutative $\mathbb{F}_q$-algebra $\mathbb{T}$. Suppose that $\mathbb{T}$ is generated by the traces of $\rho$, and that $\rho$ has big residual image, i.e. $\mathrm{Im}(\overline{\rho})=\mathrm{GL}_2^D(\mathbb{F}_q$), where $D\subseteq\mathbb{F}_q^\times$.

The explicit description that we give in Section \ref{sec_theorem} of the image of $\rho$ in this situation allows us to compute a part of a certain ray class field of the field $K$ cut out by $\overline{\rho}$ (i.e. $G_K=\ker(\overline{\rho})$). More concretely, let $\mathfrak{m}$ denote the maximal ideal of $\mathbb{T}$, assume that $\mathfrak{m}^2=0$ and let $m=\mathrm{dim}_{\mathbb{F}_q}\mathfrak{m}$.

\begin{proposition}\label{extensions}
	Let $\mathbb{F}_q$ be a finite field of characteristic $p\neq2$.
	Assume the hypothesis of Theorem $\ref{thm1}$.
	Then there exist number fields $L/K/\mathbb{Q}$ with $G_L=\ker(\rho)$ and $G_K=\ker(\overline{\rho})$ such that $\mathrm{Gal}(K/\mathbb{Q})=\mathrm{GL}_2^D(\mathbb{F}_q)$ and
	$$\mathrm{Gal}(L/\mathbb{Q})=\underbrace{\mathrm{M}_2^0(\mathbb{F}_q)\oplus\ldots\oplus\mathrm{M}_2^0(\mathbb{F}_q)}_m\rtimes\mathrm{Gal}(K/\mathbb{Q}),$$
	with $\mathrm{Gal}(K/\mathbb{Q})$ acting on $\mathrm{Gal}(L/K)$ by conjugation. 
	Moreover, the extension $L/K$ is an abelian extension of degree $p^{3dm}$ that cannot be defined over $\mathbb{Q}$ and which is unramified at all primes $\ell\nmid pN$.
\end{proposition}

\begin{proof}
	Let $\overline{G}:=\mathrm{Im}(\overline{\rho})$, $G:=\mathrm{Im}(\rho)$ and $H:=\mathrm{M}_2^0(\mathbb{F}_q)^m$. By Theorem \ref{thm1}, we have that
	$$G\simeq H\rtimes\overline{G}.$$
	This gives a short exact sequence $1\rightarrow H\rightarrow G\rightarrow\overline{G}\rightarrow1$, so $\overline{G}$ acts on $H$ by conjugation through preimages.
	Let $L:=\overline{\mathbb{Q}}^{\mathrm{ker}(\rho)}$ and $K:=\overline{\mathbb{Q}}^{\mathrm{ker}(\overline{\rho})}$. Then Galois theory tells us that we have field extensions 
	$$\xymatrix{
		& L \ar@{-}[dd]^G \ar@{-}[dl]_{H} & \\
		K\ar@{-}[dr]_{\overline{G}} & & \\
		& \mathbb{Q} &
	}$$
	The group $H$ is isomorphic to $(\mathbb{Z}/p\mathbb{Z})^{3dm}$, where $q=p^d$. Thus it is a $p$-elementary abelian group, so $L/K$ is an abelian Galois extension, unramified outside $Np$ and of degree $p^{3dm}$. Moreover, since $H$ is a simple $\mathbb{F}_p[\mathrm{GL}_2^D(\mathbb{F}_q)]$-module and the conjugation action of $\overline{G}$ on $H$ is nontrivial, we have that the extension $L/K$ cannot be defined over $\mathbb{Q}$.
\end{proof}

\begin{proposition}\label{extensions2}
	Let $\mathbb{F}_q$ be a finite field of characteristic $2$ and assume the hypothesis of Theorem $\ref{thm1}$.
	Then there are integers $0\leq\alpha\leq m$ and $0\leq\beta\leq d(m-\alpha)$ such that $t=q^\alpha\cdot((q-1)2^\beta+1)$,
	and there exists number fields $L/K/\mathbb{Q}$ with $G_L=\ker(\rho)$ and $G_K=\ker(\overline{\rho})$ such that $\mathrm{Gal}(K/\mathbb{Q})=\mathrm{GL}^D_2(\mathbb{F}_q)$ and
	$$\mathrm{Gal}(L/\mathbb{Q})\simeq M\rtimes\mathrm{Gal}(K/\mathbb{Q}),$$
	where $M$ is an $\mathbb{F}_2[\mathrm{GL}^D_2(\mathbb{F}_q)]$-submodule of $\mathrm{M}_2^0(\mathfrak{m}_\mathbb{T})$ of the form $M:=\mathrm{M}_2^0(\mathbb{F}_q)^\alpha\oplus C_2^\beta$,
	and the group $\mathrm{Gal}(K/\mathbb{Q})$ acts by conjugation on $\mathrm{Gal}(L/K)$. The extension $L/K$ is an abelian extension of degree $2^{3d\alpha+\beta}$ unramified at all primes $\ell\nmid 2N$.
	
	Moreover, if $\alpha\neq0$ there is an extension $L_1$ of $K$ with $\mathrm{Gal}(L_1/K)\simeq\mathrm{M}_2^0(\mathbb{F}_q)^\alpha$ which cannot be defined over $\mathbb{Q}$ and an extension $L_2$ of $K$ with $\mathrm{Gal}(L_2/K)=C_2^\beta$, such that $L=L_1L_2$.
\end{proposition}

\begin{proof}
	Let $\overline{G}:=\mathrm{Im}(\overline{\rho})$, $G:=\mathrm{Im}(\rho)$ and $H:=\mathrm{M}_2^0(\mathbb{F}_q)^\alpha\oplus C_2^\beta\simeq C_2^{3d\alpha+\beta}$.
	By Proposition \ref{thm2}, we have that
	$G\simeq H\rtimes\overline{G}$.
	This gives a short exact sequence $1\rightarrow H\rightarrow G\rightarrow\overline{G}\rightarrow 1$, so $\overline{G}$ acts on $H$ by conjugation through preimages.
	
	Let $L:=\overline{\mathbb{Q}}^{\mathrm{ker}(\rho)}$ and $K:=\overline{\mathbb{Q}}^{\mathrm{ker}(\overline{\rho})}$.
	Let $H_1:=\mathrm{M}_2^0(\mathbb{F}_q)^\alpha$, $H_2:= C_2^\beta$, $L_1:=L^{H_2}$ and $L_2:=L^{H_1}$.
	By Galois theory, we have field extensions
	$$
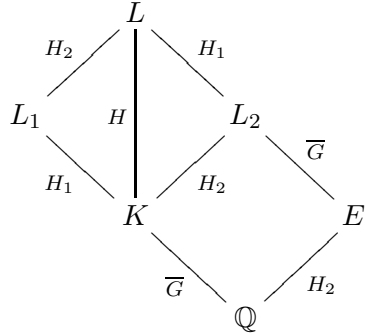
\xymatrix{
		& L \ar@{-}[dl]_{H_2}\ar@{-}[dr]^{H_1} \ar@{-}[dd]_H & & \\
		L_1\ar@{-}[dr]_{H_1} & & L_2\ar@{-}[dl]^{H_2} \ar@{-}[dr]^{\overline{G}} & \\
		& K \ar@{-}[dr]_{\overline{G}} & & E \ar@{-}[dl]^{H_2} & \\
		& & \mathbb{Q} & 
	}$$
	
	The group $H$ is an abelian group of order $2^{3d\alpha+\beta}$, and the corresponding field extension is unramified outside $2N$ because $\rho$ is.
	Moreover, since $H_2$ is contained in the centre of $G$, the action of $\overline{G}$ on $H_2$ is trivial, so $L_2=KE$, and the Galois group of the field extension $L_2/K$ is isomorphic to the Galois group of $E/\mathbb{Q}$. 
	On the other hand, $H_1$ is not contained in the centre of $G$, so the action of $\overline{G}$ on $H_1$ is nontrivial, and the extension $L_1/K$ cannot be defined over $\mathbb{Q}$.
\end{proof}
\begin{remark}
In order to show how big these extensions can be, consider one example in characteristic $p=5$. We find one fixing $N=137, k=2$ and trivial character. The corresponding Hecke algebra is defined over $\mathbb{F}_{25}$. In this case, in the notation of Proposition \ref{extensions}, this result tells us that $\mathrm{Gal}(K/\mathbb{Q})\simeq\mathrm{GL}_2(\mathbb{F}_{25})$, which has order $2^7\cdot 3^2 \cdot 5^5\cdot 13$. Assuming we can determine explicitly $K$ from that, then we would still need to compute an abelian extension of degree $5^4$ (whose existence we can predict by Proposition \ref{extensions}).
\end{remark}

\section*{Appendix: Auxiliary lemmas}\label{sec_appendix}

\begin{lemma}\label{modules}
	Fix an integer $n\geq2$ and a prime $p$ and let $k$ be a finite field with $q=p^f$ elements. Assume that $q\neq2$ if $n=2$. Let $M\subseteq\mathrm{M}_n^0(k)$ be an $\mathbb{F}_p[\mathrm{GL}^D_n(k)]$-submodule for the conjugation action. Then, either $M$ is a subspace of the scalar the matrices in $\mathrm{M}_n^0(k)$, denoted by $\mathbb{S}$, or $M=\mathrm{M}_n^0(k)$. Thus $\mathrm{M}_n^0(k)/\mathbb{S}$ is a simple $\mathbb{F}_p[\mathrm{GL}^D_n(k)]$-module, and the sequence
	$$0\rightarrow\mathbb{S}\rightarrow\mathrm{M}_n^0(k)\rightarrow T\rightarrow0$$
	does not split when $p\mid n$.
\end{lemma}

\begin{proof}
	Since any $\mathbb{F}_p[\mathrm{GL}_n^D(k)]$-submodule of $\mathrm{M}_n^0(k)$ is also an $\mathbb{F}_p[\mathrm{SL}_n(k)]$-submodule, we can apply Lemma 3.3 of \cite{Ma15}.
\end{proof}

\begin{lemma}\label{submodules}
	Let $R$ be a ring and $M$ a semisimple left $R$-module which decomposes as $M=M_1\oplus\ldots\oplus M_t$, with $M_i\subseteq M$ simple modules. Let $N\subseteq M$ be a submodule of $M$. Then $N$ is semisimple and is isomorphic to a direct sum of a subset of the modules $M_1,\ldots,M_t$.
\end{lemma}

\begin{proof}
	By Proposition 3.12 in \cite{CuRe81}, for every submodule $N\subseteq M$ there exists a submodule $N'\subseteq M$ such that $M=N\oplus N'$. Now by Corollary 14.6 in \cite{CuRe00} we have that, in this case, $N$ is isomorphic to a direct sum of a subset of the modules $M_1,\ldots, M_t$.
\end{proof}

The case of characteristic 2 is more complicated, due to the fact that the module $\mathrm{M}_2^0(\mathbb{F}_q)$ of trace-0 matrices is no longer a simple module. We will need the following lemma.

\begin{lemma}\label{p2}
	Let $k$ be a finite field of characteristic $2$ and $\# k \geq4$.
	Any $\mathbb{F}_2[\mathrm{GL}_2^D(k)]$-submodule $N$ of $\mathrm{M}_2^0(k)\oplus\ldots\oplus\mathrm{M}_2^0(k)$ is of the form
	$$N\simeq N_1\oplus\ldots\oplus N_m,$$
	where each $N_i$ is either $\mathrm{M}_2^0(k)$ or an $\mathbb{F}_2$-subspace of the scalar matrices  $\mathbb{S}\subseteq\mathrm{M}_2^0(k)$, for $1\leq i\leq m$.
\end{lemma}

\begin{proof}
	We denote $G:=\mathrm{GL}_2^D(k)$, and for any module $M$, denote by $M^n$ the direct sum $\oplus_{i=1}^nM$. Let $T$ denote the module $\mathrm{M}_2^0(k)/\mathbb{S}$, which is simple after Lemma \ref{modules}, and consider the projection $\pi:\mathrm{M}_2^0(k)^n\rightarrow T^n$.
	By Lemma \ref{submodules} we have that 
	$$\pi(N)\simeq T^\ell, \mbox{\ for\ some\ }0\leq \ell\leq n.$$
	Let $\alpha:T^\ell\rightarrow\pi(N)$ denote this isomorphism. 
	
	We will first see that we can assume, without loss of generality, that $\pi(N)$ maps onto the first $\ell$ copies of $T$ (after isomorphism).
	Consider the composition
	$$\varphi_{ij}:T\stackrel{\iota_j}{\hookrightarrow} T^\ell \stackrel{\alpha}{\rightarrow}\pi(N)\stackrel{\iota}{\hookrightarrow} T^n \stackrel{\pi_i}{\twoheadrightarrow}T,$$
	where $\iota_j$ denotes the natural inclusion of $T$ into the $j$-th component of $T^\ell$, for $1\leq j \leq \ell$, and $\pi_i$ denotes the natural projection from $T^n$ to the $i$th component $T$, for $1\leq i\leq n$.
	By Lemma 3 in \cite{Ma15} , we have that $\varphi_{ij}$ is just multiplication by a scalar in $k$, so we can write $\varphi_{ij}(t)=a_{ij}t$, for $a_{ij}\in k,\ t\in T$.
	We then can express $\varphi:=\iota\circ\alpha$ as
	$$\begin{array}{cccl}
		\varphi: & T^\ell & \rightarrow & T^n \\
		& \left(\begin{smallmatrix}t_1\\ t_2 \\ \vdots \\ t_\ell\end{smallmatrix}\right) & \mapsto &  \left(\begin{smallmatrix}a_{11} & \ldots & a_{1\ell} \\ a_{21} & \ldots & a_{2\ell} \\ \vdots & & \vdots \\ a_{n1} & \ldots & a_{n\ell}  \end{smallmatrix}\right)\left(\begin{smallmatrix}t_1\\ t_2 \\ \vdots \\ t_\ell\end{smallmatrix}\right).
	\end{array}$$
	Since $\varphi$ is injective, the matrix $A:=\left(\begin{smallmatrix}a_{11} & \ldots & a_{1\ell} \\ \vdots & & \vdots \\ a_{n1} & \ldots & a_{n\ell}  \end{smallmatrix}\right)$ has rank $\ell$. Using Gau\ss\ elimination, we can find an invertible $n\times n$ matrix $C$ such that
	$$CA=\left(\begin{smallmatrix}1 & 0 & \ldots & 0 \\ 0 & 1 & \ldots & 0 \\ \vdots & \vdots & & \vdots \\ 0 & 0 & \ldots & 1 \\ 0 & 0 & \ldots & 0 \\ \vdots & \vdots & &  \vdots \\ 0 & 0 & \ldots & 0 \end{smallmatrix}\right).$$
	Note that $\cdot C:\mathrm{M}_2^0(k)^n\rightarrow \mathrm{M}_2^0(k)^n$ and $\cdot C:T^n\rightarrow T^n$ are isomorphisms of $\mathbb{F}_2[G]$-modules. By construction we obtain
	$$C\cdot\iota(\pi(N))=C\cdot\iota\cdot\alpha(T^\ell)=C\cdot\varphi(T^\ell)=\underbrace{T\oplus\ldots\oplus T}_\ell\oplus\underbrace{0\oplus\ldots\oplus 0}_{n-\ell}.$$
	Thus, we have seen that we can assume
	$$\pi(N)=\underbrace{T\oplus\ldots\oplus T}_\ell\oplus\underbrace{0\oplus\ldots\oplus0}_{n-\ell}\subseteq T^n.$$
	We will see that $N\simeq\mathrm{M}_2^0(k)^\ell\oplus L$, where $L\subseteq \mathbb{S}^{n-\ell}$ is an $\mathbb{F}_2[G]$-submodule.
	Note that $\pi(N)=T^\ell$ is equivalent to 
	\begin{eqnarray}\label{eq1}
	N+\mathbb{S}^n=M^\ell\oplus \mathbb{S}^{n-\ell},
	\end{eqnarray}
	since $\pi(N)=\pi(M^\ell\oplus \mathbb{S}^{n-\ell})$ and $\ker(\pi)=\mathbb{S}^n$.
	
	Take $(m_1,\ldots,m_\ell, \lambda_1,\ldots, \lambda_{n-\ell})\in M^\ell\oplus\mathbb{S}^{n-\ell}$. Then equality (\ref{eq1}) tells us that there exists $x\in N$ and $(\mu_1,\ldots,\mu_n)\in\mathbb{S}^n$ such that
	$$x=(m_1,\ldots,m_\ell, \lambda_1,\ldots, \lambda_{n-\ell})+(\mu_1,\ldots,\mu_n).$$
	 In particular, if we take $m_1=\left(\begin{smallmatrix}0&1\\0&0\end{smallmatrix}\right)$, and $m_i=0$, for $2\leq i\leq \ell$, then we have that $N$ contains an element of the form
	 $$\left(\left( \begin{smallmatrix} \mu_1 & 1 \\ 0 & \mu_1 \end{smallmatrix}\right), \left( \begin{smallmatrix} \mu_2 & 0 \\ 0 & \mu_2 \end{smallmatrix}\right),\ldots,\left( \begin{smallmatrix} \mu_n & 0 \\ 0 & \mu_n \end{smallmatrix}\right)\right).$$
	 In particular, the first component of $\pi(N)$ is not 0. Since, by Lemma \ref{modules}, $\mathrm{M}_2^0(k)$ is indecomposable and all its submodules are contained in $\mathbb{S}=\ker(\pi)$, we have $M\oplus\{0\}^{n-1}\subseteq N$.
	An analogous argument for the other components gives that 
	$$M^\ell\oplus\{0\}^{n-\ell}\subseteq N\subseteq M^\ell\oplus\mathbb{S}^{n-\ell}.$$
	
	Finally, by the 4th isomorphism theorem, we have that the submodules of $M^\ell\oplus\mathbb{S}^{n-\ell}$ containing $M^\ell\oplus\{0\}^{n-\ell}$ are in bijection with the submodules of $(M^\ell\oplus\,\mathbb{S}^{n-\ell})/(M^\ell\oplus\{0\}^{n-\ell})\simeq\mathbb{S}^{n-\ell}$. 
	The module $\mathbb{S}$ is semisimple, and we know by Lemma \ref{modules} that its submodules  $S_1,\ldots, S_t$ are $\mathbb{F}_2$-subspaces of $\mathbb{S}$. The submodules of $\mathbb{S}^{n-\ell}$ are isomorphic to a direct sum of a subset of the modules $S_1,\ldots, S_t$ by Lemma \ref{submodules}. This concludes the proof.
\end{proof}

\Addresses

\end{document}